\documentclass[preprint,12pt]{elsarticle}

\usepackage{graphicx}
\usepackage[cp1251]{inputenc}
\usepackage{amssymb}
\usepackage{amsthm}
\usepackage{cmap}

\usepackage{amsmath}

\biboptions{sort&compress}

\newcommand{\sysn}{\left\{\begin{array}{rcl}}
\newcommand{\sysk}{\end{array}\right.}

\newtheorem{theorem}{Theorem}[section]
\newtheorem{lemma}[theorem]{Lemma}

\newtheorem{proposition}[theorem]{Proposition}

\newtheorem{corollary}[theorem]{Corollary}

\theoremstyle{definition}
\newtheorem{definition}[theorem]{Definition}
\newtheorem{question}[theorem]{Question}
\newtheorem{example}[theorem]{Example}
\newtheorem{remark}[theorem]{Remark}

\journal{...}

\begin{document}

\title{On resolvability and tightness in uncountable spaces}

\author{Anton E. Lipin}

\address{Krasovskii Institute of Mathematics and Mechanics, \\ Ural Federal
 University, Yekaterinburg, Russia}

\ead{tony.lipin@yandex.ru}

\begin{abstract}

We investigate connections between resolvability and different forms of tightness.
This study is adjacent to \cite{BM,BH}.

We construct a non-regular refinement $\tau^*$ of the natural topology of the real line $\mathbb{R}$ with properties such that the space $(\mathbb{R}, \tau^*)$ has a hereditary nowhere dense tightness and it has no $\omega_1$-resolvable subspaces, whereas $\Delta(\mathbb{R}, \tau^*) = \frak{c}$.

We also show that the proof of the main result of \cite{BH}, being slightly modified, leads to the following strengthening: if $L$ is a Hausdorff space of countable character and the space $L^\omega$ is c.c.c., then every submaximal dense subspace of $L^\kappa$ has disjoint tightness.
As a corollary, for every $\kappa \geq \omega$ there is a Tychonoff submaximal space $X$ such that $|X|=\Delta(X)=\kappa$ and $X$ has disjoint tightness.
\end{abstract}


\maketitle 


\section{Introduction}

\begin{definition}[\cite{Hewitt, Ceder}]
A space $X$ is called {\it $\kappa$-resolvable}, if $X$ contains $\kappa$ many pairwise disjoint dense subsets.
A space $X$ is called {\it maximally resolvable}, if $X$ is $\Delta(X)$-resolvable, where
$$\Delta(X) = \min\{|U| : U \text{ is nonempty open in } X\}$$
is the {\it dispersion character} of the space $X$.
A space $X$ is called {\it $\kappa$-irresolvable}, if $\kappa \leq \Delta(X)$ and $X$ is not $\kappa$-resolvable.
Besides that, 2-resolvable and 2-irresolvable spaces are also called {\it resolvable} and {\it irresolvable}, respectively.
\end{definition}

We refer the reader to a selective survey \cite{survey} for more details on resolvability.

Recall that {\it tightness} of a space $X$ is the smallest cardinal $\kappa$ such that for all $A \subseteq X$ and $x \in A'$ there is $D \subseteq A$ such that $|D|\leq \kappa$ and $x \in D'$.
Tightness of $X$ is denoted by $t(X)$.
E.G.~Pytkeev proved that every space $X$ such that $t(X)<\Delta(X)$ is maximally resolvable \cite{Pytkeev}.

In 1998 A.~Bella and V.I.~Malykhin investigated connections between resolvability and other forms of tightness in the class of countable spaces \cite{BM}.
The following notion is used in one of their positive results.

\begin{definition}
A space $X$ is said to have {\it nowhere dense tightness} (let us abbreviate it as {\it NDT}), if whenever $A \subseteq X$ and $x \in A'$ there is a set $N \subseteq A$ such that $N$ is nowhere dense in $X$ and $x \in N'$.
\end{definition}

\begin{theorem}[\cite{BM}]
Every countable crowded space $X$ with NDT is $\omega$-resolvable.
\end{theorem}

However, the following question remains open.

\begin{question}
Is every crowded space with NDT resolvable?
\end{question}

There is a form of tightness which is stronger than NDT and has well studied connections with resolvability.
Recall that a set $D$ in a space $X$ is called {\it strongly discrete}, if there are pairwise disjoint open sets $U_x \ni x$ for all $x \in D$.

\begin{definition}
A space $X$ is said to have {\it strongly discrete tightness}, if whenever $A \subseteq X$ and $x \in A'$, there is a set $D \subseteq A$ such that $D$ is strongly discrete and $x \in D'$.
\end{definition}

Clearly, strongly discrete tightness entails NDT in the class of crowded spaces.

\begin{theorem}[\cite{JSSz_MN}]
Every crowded space with strongly discrete tightness is $\omega$-resolvable.
\end{theorem}

\begin{theorem}[\cite{JSSz_MN}]
If $\kappa$ is a measurable cardinal, then there is a normal $\omega_1$-irresolvable space with strongly discrete tightness and dispersion character $\kappa$.
\end{theorem}

\begin{theorem}[\cite{JM}]
The following conditions are equiconsistent:
\begin{enumerate}

\item[(1)] there is a measurable cardinal;

\item[(2)] there is a normal $\omega_1$-irresolvable space with strongly discrete tightness;

\item[(3)] there is a space with strongly discrete tightness which is not maximally resolvable.

\end{enumerate}
\end{theorem}

In this paper we construct in ZFC a refinement $\tau^*$ of the natural topology of the real line $\mathbb{R}$ with the following properties:
$\Delta(\mathbb{R}, \tau^*)=\frak{c}$, the space $(\mathbb{R}, \tau^*)$ has hereditary NDT and it is hereditary $\omega_1$-irresolvable (Example \ref{ex_ndt}). The space $(\mathbb{R}, \tau^*)$ is not regular.

This example is constructed via a technique that we shall call XEI-processing.
Its basic notions and general properties are established in Section 3, whereas the specifics of the example are presented in Section 4.

Section 5 is devoted to another form of tightness, namely:

\begin{definition}[\cite{BM}]
A space $X$ is said to have {\it disjoint tightness}, if whenever $A \subseteq X$ and $x \in A'$, there are $A_1, A_2 \subseteq A$ such that $A_1 \cap A_2 = \emptyset$ and $x \in A_1' \cap A_2'$.
\end{definition}

In \cite{BM} Bella and Malykhin questioned whether there is a countable regular irresolvable space with disjoint tightness.
They also constructed an example of such a space under CH.
In \cite{BH} Bella and M.~Hru{\v s}{\'a}k proved the positive answer in ZFC.
Furthermore, they showed that such a space can be submaximal, i.e. with the property that all its dense subspaces are open.
Clearly, any crowded submaximal space is irresolvable.

\begin{theorem}[\cite{BH}]
\label{T_intro_BH}
If $X$ is a countable dense submaximal subspace of some Cantor cube $2^\kappa$, then $X$ has disjoint tightness.
\end{theorem}

The premise of this theorem is not always false. In particular, in \cite{JSSz_D} I. Juh{\'a}sz, L. Soukup and Z. Szentmikl{\'o}ssy proved that for every $\lambda \geq \omega$ there is a dense submaximal subspace $X$ of the Cantor cube $2^{2^\lambda}$ such that $|X|=\Delta(X)=\lambda$.
Besides that, C.~Corral proved that there is a model of ZFC with a countable dense submaximal subspace of $2^{\omega_1}$, whereas $\frak{c} = \omega_2$ \cite{Corral}.

By carefully studying the proof of Theorem \ref{T_intro_BH}, one can note that countability of $X$ is not actually essential here.
We show it in Section 5 (Theorem \ref{T_BH}). We also decompose the proof to highlight some noteworthy properties and observations.
Moreover, we consider not only the class of Cantor cubes, but also some wider class of topological cubes (including Tychonoff cubes).
This generalization may be interesting because of the following two theorems.

\begin{theorem}[\cite{HPST}, Corollary 4]
	Let $M$ be a separable metric space consisting of more than one point.
	For each infinite $\kappa$, there is a dense submaximal normal (hence perfectly normal) subset $X \subseteq M^{2^\kappa}$ of cardinality $\kappa$.
\end{theorem}

\begin{theorem}[\cite{Lipin}]
If $L$ is a $T_1$ space, $|L|>1$ and $d(L) \leq \kappa \geq \omega$, then there is a submaximal dense subspace $X$ of $L^{2^\kappa}$ such that $|X|=\Delta(X)=\kappa$.
\end{theorem}

At the same time, the idea of the proof of Theorem \ref{T_BH} is almost identical to the proof of the original Theorem \ref{T_intro_BH}.

\section{Preliminaries}

We assume the following notation and conventions.

\begin{itemize}

\item If $X$ is a set and $\kappa$ is a cardinal, then
$[X]^{\kappa} = \{A \subseteq X : |A| = \kappa\}$ and $[X]^{<\kappa} = \{A \subseteq X : |A| < \kappa\}$.

\item Symbol $\bigsqcup$ denotes {\it disjoint union} in the following sense: it is equal to the usual union, but using it we assume that the united sets are pairwise disjoint.

\item {\it Space} means topological space.

\item The dispersion character of a space $(X,\tau)$ can be denoted by $\Delta(X)$, $\Delta(X, \tau)$ or $\Delta(\tau)$.

\end{itemize}

We also have to recall that a space $X$ is called {\it submaximal}, if all its dense subsets are open, and we also have to recall the following criterions.

\begin{proposition}
\label{P_pre_submax}
For every space $X$ the following conditions are equivalent:
\begin{enumerate}

\item[(1)] $X$ is submaximal;

\item[(2)] for every $A \subseteq X$ if the interior of $A$ is empty, then $A$ has no accumulation points.

\item[(3)] for every $A \subseteq X$ and $x \in A'$ there is open $U \subseteq A$ such that $x \in U'$.

\end{enumerate}
\end{proposition}

It is easy to see that every crowded submaximal space is irresolvable.

\section{XEI-processes}

Recall that a family $\mathcal{I}$ of subsets of a set $X$ is called an {\it ideal} on $X$, if all of the following conditions hold:

\begin{itemize}

\item[(1)] for all $S \in \mathcal{I}$ and $T \subseteq S$ we have $T \in \mathcal{I}$;

\item[(2)] for all $S,T \in \mathcal{I}$ we have $S \cup T \in \mathcal{I}$;

\item[(3)] $\emptyset \in \mathcal{I}$ and $X \notin \mathcal{I}$.

\end{itemize}

\begin{definition}
We say that a tuple $(X, \tau, \mathcal{E}, \mathcal{I}, \mathcal{I}^*, \tau^*)$ is an {\it XEI-process}, if all of the following conditions are satisfied:

\begin{itemize}

\item[(a)] $(X, \tau)$ is a space;

\item[(b)] $\mathcal{E}$ is a family of nonempty subsets of $X$ and $\mathcal{E} \supseteq \tau \setminus \{\emptyset\}$;

\item[(c)] if $K \in \mathcal{E}$ and $U \in \tau$, then $K \cap U \in \mathcal{E}\cup\{\emptyset\}$;

\item[(d)] $\mathcal{I}$ is an ideal on $X$. Moreover, $\mathcal{I} \supseteq [X]^{<\omega}$ and $\mathcal{I} \cap \mathcal{E} = \emptyset$;

\item[(e)] $\mathcal{I}^* \supseteq \mathcal{I}$ is an ideal on $X$, and it is maximal with the condition $\mathcal{I}^* \cap \mathcal{E} = \emptyset$;

\item[(f)] $\tau^*$ is the topology on $X$ with the base $\{U \setminus S : U \in \tau, S \in \mathcal{I}^*\}$.

\end{itemize}
\end{definition}

It is easy to see that every tuple $(X,\tau, \mathcal{E}, \mathcal{I})$ satisfying (a-d) can be extended to some XEI-process (by Zorn lemma).
Let us denote by (XEI) the condition ``$(X, \tau, \mathcal{E}, \mathcal{I}, \mathcal{I}^*, \tau^*)$ is an XEI-process''.
If (XEI), then for every $A \subseteq X$ we denote by $A^*$ the set of all accumulation points of $A$ in the topology $\tau^*$.

Let us note some general properties of XEI-processes.

\begin{proposition}
\label{P_XEI}
If {\rm (XEI)}, then all of the following statements hold:

\begin{enumerate}

\item[(1)] $\tau^*$ is a refinement of the topology $\tau$;

\item[(2)] if $S \in \mathcal{I}^*$, then $S^* = \emptyset$;

\item[(3)] $A \notin \mathcal{I}^*$ iff there is a set $K \in \mathcal{E}$ such that $K \setminus A \in \mathcal{I}^*$ (for all $A \subseteq X$);

\item[(4)] $K^* = K'$ for all $K \in \mathcal{E}$;

\item[(5)] if $\mathcal{I} \supseteq [X]^{<\Delta(\tau)}$, then $\Delta(\tau^*) = \Delta(\tau)$.

\end{enumerate}
\end{proposition}
\begin{proof}
(1) is obvious.

(2). First note that all $S \in \mathcal{I}^*$ are closed, because $X \setminus S$ are open.
Since $\mathcal{I}^*$ is an ideal, all such sets $S$ are hereditary closed, so $S^* = \emptyset$.

(3) follows from maximality of $\mathcal{I}^*$.

(4). It follows from (1) that $K^* \subseteq K'$. Let us take any point $x \in K'$ and show that $x \in K^*$.
Take any neighborhood $U \setminus S$ of the point $x$, where $U \in \tau$ and $S \in \mathcal{I}^*$.
Since $x \in K'$, we have $K \cap U \ne \emptyset$, so $K \cap U \in \mathcal{E}$.
Consequently, $K \cap U \notin \mathcal{I}^*$, so $K \cap U \setminus (S \cup \{x\}) \ne \emptyset$.

(5). Take any nonempty set $U \setminus S$, where $U \in \tau$ and $S \in \mathcal{I}^*$. Since $U \in \tau \setminus \{\emptyset\} \subseteq \mathcal{E}$, we have $U \notin \mathcal{I}^*$, so $U \setminus S \notin \mathcal{I}^*$, hence $|U \setminus S| \geq \Delta(\tau)$.
\end{proof}

Before we get into more complicated properties of XEI-processes, let us show a simple example.

\begin{example}
If {\rm (XEI)} and $\mathcal{E} = \tau \setminus \{\emptyset\}$, then the space $(X, \tau^*)$ is submaximal.
\end{example}
\begin{proof}
Suppose $A^* \ne \emptyset$ for some $A \subseteq X$. Then $A \notin \mathcal{I}^*$, so there is $K \in \mathcal{E}$ such that $K \setminus A \in \mathcal{I}^*$. Thus, the set $A$ contains a nonempty open subset $K \cap A$.
\end{proof}

\begin{lemma}
\label{XEI_acc}
If {\rm (XEI)}, then for every set $A \subseteq X$ and any point $x \in X$ the following conditions are equivalent:
\begin{enumerate}

\item[(1)] $x \in A^*$;

\item[(2)] whenever $x \in U \in \tau$, there is a set $K \in \mathcal{E}$ such that $K \subseteq U$ and $K \setminus A \in \mathcal{I}^*$.

\end{enumerate}
\end{lemma}
\begin{proof}
(1) $\to$ (2). Suppose that $x \in U \in \tau$. Since $x \in (A \cap U)^* \ne \emptyset$, we can conclude that $A \cap U \notin \mathcal{I}^*$.
It follows that there is a set $L \in \mathcal{E}$ such that $L \setminus (A \cap U) \in \mathcal{I}^*$. Denote $K = L \cap U$.
It is easy to see that $K$ is as required.

(2) $\to$ (1). Take any neighborhood $U \setminus S$ of the point $x$, where $U \in \tau$ and $S \in \mathcal{I}^*$.
Denote $T = S \cup \{x\}$ and $H = A \cap U \setminus T$.
We have to prove that $H \ne \emptyset$.

Take $K \in \mathcal{E}$ such that $K \subseteq U$ and $K \setminus A \in \mathcal{I}^*$.
Denote $R = T \cup (K \setminus A)$. Clearly, $R  \in \mathcal{I}^*$, so $K \setminus R \ne \emptyset$.
It remains to note that $K \setminus R \subseteq H$.
\end{proof}

Suppose $\kappa$ is a cardinal.
Recall that an ideal $\mathcal{I}$ on a set $X$ is called $\kappa$-{\it saturated}, if for every family $\mathcal{A}$ of subsets of $X$ such that $|\mathcal{A}|=\kappa$ and $A \cap B \in \mathcal{I}$ for all distinct $A,B \in \mathcal{A}$ we have $\mathcal{A} \cap \mathcal{I} \ne \emptyset$.

Let us say that a family of sets $\mathcal{E}$ satisfies the property {\rm (A$\empty_\kappa$)}, if for every $\mathcal{H} \in [\mathcal{E}]^\kappa$ there are different $K,L \in \mathcal{H}$ and $M \in \mathcal{E}$ such that $K \cap L \supseteq M$.

\begin{lemma}
\label{XEI_irres}
If {\rm (XEI)} and the family $\mathcal{E}$ satisfies {\rm (A$\empty_\kappa$)}, then all the following statements hold:

\begin{enumerate}
\item[(1)] the ideal $\mathcal{I}^*$ is $\kappa$-saturated;

\item[(2)] if $\bigsqcup\limits_{\alpha<\kappa} A_\alpha \subseteq X$, then there is $\alpha$ such that $A_\alpha^*=\emptyset$;

\item[(3)] the space $(X,\tau^*)$ is hereditary $\kappa$-irresolvable.
\end{enumerate}
\end{lemma}
\begin{proof}
(1). Suppose, on the contrary, that there is a family $\mathcal{A}$ of subsets of $X$ such that $|\mathcal{A}|=\kappa$, $A \cap B \in \mathcal{I}^*$ for all different $A,B \in \mathcal{A}$ and $\mathcal{A} \cap \mathcal{I}^* = \emptyset$.
Thus, for every $A \in \mathcal{A}$ there is a set $K_A \in \mathcal{E}$ such that $K_A \setminus A \in \mathcal{I}^*$.
By the condition (A$\empty_\kappa$) there are distinct $A,B \in \mathcal{A}$ and $M \in \mathcal{E}$ such that $K_A \cap K_B \supseteq M$.
On the other hand, $K_A \cap K_B \subseteq (K_A \setminus A) \cup (K_B \setminus B) \cup (A \cap B) \in \mathcal{I}^*$.
It is a contradiction.

(2) follows from (1).
(3) follows from (2).
\end{proof}

\section{An example of $\omega_1$-irresolvable space with NDT}

In this section we use the following notation:

\begin{itemize}

\item $(\mathbb{R}, \tau_\mathbb{R})$ is the real line with its natural topology;

\item $\mu$ is the Lebesgue measure on $\mathbb{R}$;

\item $\mathcal{L}$ is the family of all nonempty measurable sets $K \subseteq \mathbb{R}$ such that for every $U \in \tau$ either $K \cap U = \emptyset$ or $\mu(K \cap U) > 0$.

\end{itemize}

We shall get the required example $(\mathbb{R}, \tau^*)$ by extending the tuple $(\mathbb{R}, \tau_\mathbb{R}, \mathcal{L}, [\mathbb{R}]^{<\frak{c}})$ to an XEI-process.
At first, we need to point out some properties (no doubt, well-known) of the family $\mathcal{L}$.

\begin{proposition}
\label{P_L}
If $A \subseteq \mathbb{R}$ and $\mu(A)>0$, then the set $A$ contains an element of the family $\mathcal{L}$.
\end{proposition}
\begin{proof}
Denote by $\mathcal{H}$ the family of all intervals $U \subseteq \mathbb{R}$ such that both ends of $U$ are rational and $\mu(A \cap U) = 0$.
Denote $H = \bigcup \mathcal{H}$. Clearly, $\mu(A \cap H) = 0$. Denote $K = A \setminus H$. It is easy to see that $K \in \mathcal{L}$.
\end{proof}

\begin{proposition}
\label{P_LNDL}
For every $K \in \mathcal{L}$ there is a set $L \subseteq K$ such that $L \in \mathcal{L}$ and $L$ is nowhere dense in $K$ (in the topology $\tau_\mathbb{R}$).
\end{proposition}
\begin{proof}
Choose any countable dense subset $Q = \{q_n : n \in \mathbb{N}\}$ of the set $K$.
For all $n \in \mathbb{N}$ denote $U_n = (q_n - \frac{1}{4^n}\mu(K), q_n + \frac{1}{4^n}\mu(K))$.
Construct $U = \bigcup\limits_{n \in \mathbb{N}} U_n$.
Obviously, $\mu(U) \leq \frac{2}{3}\mu(K)$, so $\mu(K \setminus U)>0$.
Finally, by the previous Proposition \ref{P_L} there is $L \subseteq K \setminus U$ such that $L \in \mathcal{L}$.
It is clear that $L$ is nowhere dense in $K$.
\end{proof}

\begin{proposition}\label{P_LA}
	The family $\mathcal{L}$ satisfies the property {\rm (A$\empty_{\omega_1}$)}.
\end{proposition}
\begin{proof}
	Suppose $\mathcal{H} \subseteq \mathcal{L}$ and for all distinct $K,L \in \mathcal{H}$ the intersection $K \cap L$ contains no element of $\mathcal{L}$.
	It follows from Proposition \ref{P_L} that $\mu(K \cap L) = 0$.
	For every $n,k \in \mathbb{N}$ let us denote $\mathcal{H}_n^k = \{K \in \mathcal{H} : \mu(K \cap [-n, n]) > \frac{1}{k}\}$.
	Clearly, $\mathcal{H} = \bigcup_{n,k \in \mathbb{N}} \mathcal{H}_n^k$.
	For all $K_1,\ldots,K_m \in \mathcal{H}$ and $n \in \mathbb{N}$ we have $\mu(\bigcup_{i=1}^m K_i \cap [-n, n]) = \sum_{i=1}^m \mu(K_i \cap [-n, n])$.
	It follows that every $\mathcal{H}_n^k$ is finite, so $|\mathcal{H}| \leq \omega$.
\end{proof}

We shall show that our example has hereditary NDT. Let us make an obvious observation.

\begin{proposition}
For every space $X$ the following properties are equivalent:

\begin{enumerate}

\item[(1)] all subspaces of $X$ have NDT;

\item[(2)] whenever $A \subseteq X$ and $x \in A'$, there is a set $N \subseteq A$ such that $N$ is nowhere dense in $A$ and $x \in N'$.

\end{enumerate}
\end{proposition}

\begin{example}
\label{ex_ndt}
A refinement $\tau^*$ of the topology $\tau_\mathbb{R}$ such that $\Delta(\tau^*) = \frak{c}$, the space $(\mathbb{R}, \tau^*)$ has hereditary NDT and it is hereditary $\omega_1$-irresolvable (furthermore, whenever $\bigsqcup\limits_{\alpha<\omega_1} A_\alpha \subseteq \mathbb{R}$, at least one of the sets $A_\alpha$ has no accumulation points).
\end{example}
\begin{proof}
As it was told, we take the tuple $(\mathbb{R}, \tau_\mathbb{R}, \mathcal{L}, [\mathbb{R}]^{<\frak{c}})$ and extend it to some XEI-process
$(\mathbb{R}, \tau_\mathbb{R}, \mathcal{L}, [\mathbb{R}]^{<\frak{c}}, \mathcal{I}^*, \tau^*)$.

The property $\Delta(\tau^*)=\frak{c}$ follows from Proposition \ref{P_XEI}(5).
Hereditary $\omega_1$-irresolvability (and its declared stronger version) follows from Lemma \ref{XEI_irres} and Proposition \ref{P_LA}.
It remains to prove hereditary NDT.

Suppose we are given a set $A \subseteq \mathbb{R}$ and a point $x \in A^*$.

Recursively by $n \in \mathbb{N}$ let us choose intervals $U_n \subseteq \mathbb{R}$ and sets $K_n \in \mathcal{L}$ in such a way that
$U_n \subseteq (x - \frac{1}{n}, x + \frac{1}{n})$,
$x \notin U_n'$,
$U_n$ are pairwise disjoint,
$K_n \subseteq U_n$ and
$K_n \setminus A \in \mathcal{I}^*$.

To see that it is possible, choose an interval $V_n \subseteq \mathbb{R}$ such that $x \in V_n \subseteq (x - \frac{1}{n}, x + \frac{1}{n})$ and $V_n \cap U_k = \emptyset$ for all $k<n$,
then apply Lemma \ref{XEI_acc} to obtain a set $M_n \in \mathcal{L}$ such that $M_n \subseteq V_n$ and $M_n \setminus A \in \mathcal{I}^*$,
then take an interval $U_n \subseteq V_n$ such that $x \notin U_n'$ and $\mu(M_n \cap U_n) > 0$, and finally define $K_n = M_n \cap U_n$.

Now by Proposition \ref{P_LNDL} there are sets $L_n \subseteq K_n$ such that $L_n \in \mathcal{L}$ and $L_n$ is nowhere dense in $K_n$ in the topology $\tau_\mathbb{R}$.
Denote $L = \bigcup\limits_{n \in \mathbb{N}} L_n$ and $N = L \cap A$.

Let us show that $N$ is nowhere dense in $A$ (in $\tau^*$).
Take any $U \in \tau$ and $S \in \mathcal{I}^*$ such that $N \cap U \setminus S \ne \emptyset$.
Clearly, we can suppose that $U \subseteq U_n$ for some $n \in \mathbb{N}$, so $L_n \cap A \cap U \setminus S \ne \emptyset$.
Since $L_n$ is nowhere dense in $K_n$ (in $\tau_\mathbb{R}$),
there is an interval $V \subseteq U$ such that $L_n \cap V = \emptyset$ and still $K_n \cap V \ne \emptyset$ (so $K_n \cap V \in \mathcal{L}$).
Thus, for the set $V \setminus S \subseteq U \setminus S$ we have $N \cap V \setminus S = \emptyset$, whereas $A \cap V \setminus S \ne \emptyset$.
It follows that $N$  is nowhere dense in $A$ in the topology $\tau^*$.

Finally, by Lemma \ref{XEI_acc} we have $x \in N^*$, because whenever $x \in U \in \tau$ there is $L_n$ such that $L_n \subseteq U$.
\end{proof}

\begin{remark}
The space $(\mathbb{R}, \tau^*)$ from Example \ref{ex_ndt} is not regular.
For instance, there is no neighborhood $V$ of the point $0$ such that $\overline{V} \subseteq (-1, 1) \setminus \{\frac{1}{n} : n \in \mathbb{N}\} \in \tau^*$.
\end{remark}

One more observation on the space $(\mathbb{R}, \tau^*)$ requires the following

\begin{lemma}\label{LMeasRes}
	For every $K \in \mathcal{L}$ there are sets $L_n \in \mathcal{L}$, $n \in \omega$, such that $\bigsqcup_{n \in \omega} L_n \subseteq K$ and every $L_n$ is dense in $K$ in the topology $\tau_\mathbb{R}$.
\end{lemma}
\begin{proof}
	Denote by $\mathcal{U}$ the family of all intervals $(a,b) \subseteq \mathbb{R}$ such that $a,b \in \mathbb{Q}$ and $K \cap (a,b) \ne \emptyset$.
	Enumerate $\mathcal{U} \times \omega = \{(U_i, n_i) : i \in \omega\}$.
	Recursively on $i \in \omega$, choose pairwise disjoint sets $M_i \subseteq K \cap U_i$ such that $M_i \in \mathcal{L}$ and all $M_i$ are nowhere dense in $K$ (which is possible by Proposition \ref{P_LNDL}).
	Define $L_n$ as the union of all $M_i$ such that $n_i = n$.
	It is clear that the sets $L_n$ are as required.
\end{proof}

Also let us recall

\begin{proposition}[\cite{Ceder}, Theorem 4]\label{PCeder}
	A space $X$ is $\kappa$-resolvable iff every its nonempty open subset contains a nonempty $\kappa$-resolvable subspace.
\end{proposition}

\begin{remark}
All crowded subspaces of the space $(\mathbb{R}, \tau^*)$ are $\omega$-resolvable.
\end{remark}
\begin{proof}
	By Proposition \ref{PCeder} it is enough to prove that every nonempty crowded subspace of $(\mathbb{R}, \tau^*)$ contains an $\omega$-resolvable subspace.
	Suppose a nonempty set $X \subseteq \mathbb{R}$ is crowded in the topology $\tau^*$.
	In particular, $X \notin \mathcal{I}^*$, so there is a set $K \in \mathcal{L}$ such that $K \setminus X \in \mathcal{I}^*$.
	We will prove that the subspace $K \cap X$ is $\omega$-resolvable.	
	
	By Lemma \ref{LMeasRes} there are sets $L_n \in \mathcal{L}$ such that $\bigsqcup_{n \in \omega} L_n \subseteq K$ and every $L_n$ is dense in $K$ in the topology $\tau_\mathbb{R}$. By Proposition \ref{P_XEI}(4) it follows that every $L_n$ is dense in $K$ in the topology $\tau^*$ as well.
	Since $L_n \setminus X \subseteq K \setminus X \in \mathcal{I}^*$, we have $(L_n \cap X)^* = L_n^*$, so every $L_n \cap X$ is dense in $K \cap X$.
\end{proof}

\section{Disjoint tightness and submaximality}

The technique of this section generally repeats the technique in \cite{BH}, but we obtain a noticeably stronger result.

\begin{definition}
We say that a space $X$ has {\it open disjoint tightness} ({\it ODT}) if for every open set $U \subseteq X$ and any point $x \in U'$ there are open sets $U_1,U_2 \subseteq U$ such that $U_1 \cap U_2 = \emptyset$ and $x \in U_1' \cap U_2'$.
\end{definition}

\begin{proposition}
\label{PHFU}
Every Hausdorff Frechet-Urysohn space has ODT.
\end{proposition}

Let us fix some notation for restrictions. If $x : \kappa \to L$ and $M \subseteq \kappa$, then we denote by $x|_M$ the restriction of $x$ to $M$.
Moreover, if $A \subseteq L^\kappa$, then we denote $A|_M = \{x|_M : x \in A\}$.

\begin{definition}
Suppose we are given a space $L$ and a cardinal $\kappa$.
We say that a set $M \subseteq \kappa$ {\it determines} a set $A \subseteq L^\kappa$, if for all $x \in A$ and $y \in L^\kappa \setminus A$ we have $x|_M \ne y|_M$.
\end{definition}

\begin{proposition}
If a set $A \subseteq L^\kappa$ is determined by a set $M \subseteq \kappa$, then for every point $x \in L^\kappa$ the conditions
$x \in \overline{A}$ and $x|_M \in \overline{A|_M}$ are equivalent.
\end{proposition}

\begin{proposition}
If a space $L^\omega$ is c.c.c., then for every cardinal $\kappa$ and any open $U \subseteq L^\kappa$ there is an open set $V \subseteq U$ which is dense in $U$ and determined by a countable set.
\end{proposition}
\begin{proof}
Denote by $\mathcal{W}$ any maximal disjoint family of subsets of $U$ from the standard Tychonoff base of $L^\kappa$, i.e. of the form $W=\prod\limits_{\alpha<\kappa} W_\alpha$, where all $W_\alpha$ are open in $L$ and the set $\Gamma_W = \{\alpha<\kappa : W_\alpha \ne L\}$ is finite.
Clearly, the set $V=\bigsqcup \mathcal{W}$ is dense in $U$.
Since $L^\omega$ is c.c.c., the space $L^\kappa$ is c.c.c. too 
, so $|\mathcal{W}|\leq\omega$.
It remains to note that the set $V$ is determined by the set $\bigcup\limits_{W \in \mathcal{W}} \Gamma_W$, which is countable.
\end{proof}

\begin{lemma}
If a Hausdorff space $L^\omega$ has a countable character and c.c.c., then for every cardinal $\kappa$ the space $L^\kappa$ has ODT.
\end{lemma}
\begin{proof}
Suppose $x \in \overline{U} \setminus U$ for some open $U \subseteq L^\kappa$.
Take any open $V \subseteq U$ which is dense in $U$ and is determined by a countable set $M$.
We have $x|_M \in \overline{V|_M} \setminus V|_M$.

Since the space $L^M$ has a countable character, we can apply Proposition \ref{PHFU} and obtain open sets $U_1, U_2 \subseteq V$ such that $U_1|_M \cap U_2|_M = \emptyset$ and $x|_M \in (U_1|_M)' \cap (U_2|_M)'$. It follows that $U_1 \cap U_2 = \emptyset$ and $x \in U_1' \cap U_2'$.
\end{proof}

\begin{proposition}
Every dense subspace of a space with ODT has ODT.
\end{proposition}

The following observation easily follows from Proposition \ref{P_pre_submax}.

\begin{proposition}
For submaximal spaces disjoint tightness and ODT are equivalent.
\end{proposition}

Combining the last three statements, we obtain

\begin{theorem}
\label{T_BH}
Suppose $L$ is a Hausdorff space of countable character, the space $L^\omega$ is c.c.c. and $X$ is a submaximal dense subspace of $L^\kappa$.
Then $X$ has disjoint tightness.
\end{theorem}

Recall that for every $T_1$ space $L$ such that $|L|>1$ and $d(L) \leq \kappa \geq \omega$ there is a dense submaximal subspace $X$ in $L^{2^\kappa}$ such that $|X|=\Delta(X)=\kappa$ \cite[Theorem 3.3]{Lipin}.

\begin{corollary}
For every $\kappa \geq \omega$ there is a Tychonoff submaximal (hence irresolvable) space $X$ such that $|X|=\Delta(X)=\kappa$ and $X$ has disjoint tightness.
\end{corollary}

\begin{remark}
It follows from c.c.c. and Proposition \ref{P_pre_submax} that any space $X$ from Theorem \ref{T_BH} also has the following property: whenever $\bigsqcup\limits_{\alpha<\omega_1} A_\alpha \subseteq X$, there is $\alpha$ such that $A_\alpha$ is closed discrete.
\end{remark}

\section{A few notes on open disjoint tightness}

An example of a countable regular space that has ODT and does not have disjoint tightness can be found in \cite[Example 1.8]{BH}.

\begin{question}
\label{Q_odt_1}
Is there a space that has disjoint tightness and does not have ODT?
\end{question}

Let us show that some known spaces do not have ODT.

Recall that a space $X$ is called {\it maximal}, if it is crowded and its topology has no crowded proper refinement \cite{Hewitt}.
A space $X$ is called {\it perfectly disconnected} if no point of $X$ is a limit point of two disjoint subsets of $X$ \cite{Douwen}.
It is known that a crowded space $X$ is maximal iff it is perfectly disconnected \cite{Douwen}.
As a corollary, no maximal space has ODT.

\begin{example}
A Hausdorff compact space without ODT.
\end{example}
\begin{proof}
Take any Tychonoff maximal space $X$ (an example of such space was constructed in \cite{Elkin}).
Denote by $K$ any compactification of $X$.
Since ODT is dense-hereditary and $X$ does not have ODT, then the space $K$ does not have ODT as well.
\end{proof}


\medskip

\noindent {\bf Acknowledgements.}
The author is grateful to Maria~A. Filatova for constant attention to this work and to Alexander~V. Osipov for useful remarks.

\bibliographystyle{model1a-num-names}
\bibliography{<your-bib-database>}

\end{document}